\documentclass{amsart}
\usepackage{color,amsmath,amsthm,amsfonts,amssymb,latexsym}

\input xy
\xyoption{all}

\newcommand{\w}{\omega}
\newcommand{\F}{\mathcal F}
\newcommand{\N}{\mathcal N}

\newcommand{\A}{\mathcal A}
\newcommand{\IR}{\mathbb R}

\newcommand{\cov}{\mathrm{cov}}
\newcommand{\IN}{\mathbb N}
\newcommand{\IZ}{\mathbb Z}

\newcommand{\I}{\mathcal I}
\newcommand{\M}{\mathcal M}

\newcommand{\non}{\mathrm{non}}
\newcommand{\uhr}{\upharpoonright}

\newcommand{\Iccc}{\mathcal I_{ccc}}
\newcommand{\fIcccc}{\widetilde{\mathcal I}_{cccc}}
\newcommand{\Aut}{\mathrm{Aut}}
\newcommand{\K}{\mathcal K}
\newcommand{\IQ}{\mathbb Q}

\newcommand{\seq}[2]{\left\langle\, #1 : #2\,\right\rangle}
\newcommand{\set}[2]{\left\{ \, #1 : #2\,\right\} }

\newtheorem{theorem}{Theorem}
\newtheorem{proposition}{Proposition}

\newtheorem{problem}{Problem}
\theoremstyle{definition}
\newtheorem{remark}{Remark}
\newtheorem{example}{Example}

\title{On critical cardinalities related to $Q$-sets}
\author{Taras Banakh, Micha\l\ Machura and Lubomyr Zdomskyy}

\address{Department of Mathematics, Ivan Franko National University of Lviv, Ukraine; and\newline
Instytut Matematyki, Jan Kochanowski University in Kielce, Poland}
\email{t.o.banakh@gmail.com}
\address{Department of Mathematics, Bar-Ilan University, Ramat Gan 5290002, Israel; and \newline
Institute of Mathematics, University of Silesia, Bankowa 14, Katowice, Poland}
\email{machura@math.biu.ac.il}

\address{Kurt G\"odel Research Center for Mathematical Logic,
University of Vienna, W\"ahringer Stra\ss e 25, A-1090 Wien,
Austria}
\email{lzdomsky@gmail.com}
\subjclass{03E17, 54H05}
\keywords{$Q$-set, small uncountable cardinal}

\thanks{The third author would
like to thank the Austrian Academy of Sciences (APART Program) as
well as the Austrian Science Fund FWF (Grant I 1209-N25)
 for generous support for this research. We are also grateful to the anonymous
referee for the suggestion to include Theorem~\ref{machura}.}

\begin{document}
\begin{abstract} In this note we collect some known information
and prove  new results about
 the small uncountable cardinal $\mathfrak q_0$. The cardinal $\mathfrak q_0$ is defined as the smallest
cardinality $|A|$ of a subset $A\subset \IR$ which is not a $Q$-set (a subspace $A\subset\IR$ is called a {\em $Q$-set} if each subset $B\subset A$ is of type $F_\sigma$ in $A$).
We present a simple proof of a folklore fact that
$\mathfrak p\le\mathfrak q_0\le\min\{\mathfrak b,\non(\mathcal N),\log(\mathfrak c^+)\}$,
and also establish the consistency of a number of
strict inequalities between
 the cardinal $\mathfrak q_0$ and other standard small uncountable cardinals.
% $\mathfrak s$, $\mathfrak g$,
% $\mathfrak h$, $\add(\mathcal N)$.
This is done by
combining some  known forcing results.
%Our proof of inequality $\mathfrak p\le\mathfrak q_0$ is direct and do not use
% Bell's characterization of the cardinal $\mathfrak{p}$ as previous proofs  did.
A new result of the paper is the consistency of $\mathfrak{p} < \mathfrak{lr} < \mathfrak{q}_0$, where $\mathfrak{lr}$ denotes the linear refinement number. Another new result is the upper bound $\mathfrak q_0\le\non(\I)$ holding for any $\mathfrak q_0$-flexible cccc $\sigma$-ideal $\I$ on $\IR$.
\end{abstract}
\maketitle

A subset $X$ of the real line is called a {\em $Q$-set} if each subset $A\subset X$ is relative
$F_\sigma$-set in $A$, see \cite[\S4]{Miller}. The study of $Q$-sets was initiated by founders of Set-Theoretic Topology: Hausdorff \cite{Haus}, Sierpi\'nski \cite{Sierp} and Rothberger \cite{Roth}. $Q$-Sets are important as they appear naturally in problems related to (hereditary) normal or $\sigma$-discrete spaces; see \cite{BKT}, \cite{Burke}, \cite{Fle78}, \cite{Fle83}, \cite{FM}, \cite{GLM}, \cite{GN}, \cite{He}, \cite{HH}, \cite{Jun}, \cite{Przym}, \cite{Rudin}, \cite{Tall69}, \cite{Tall}.

We shall be interested in two critical cardinals related to $Q$-sets:
\begin{itemize}
\item $\mathfrak q_0=\min\{|X|:X\subset\IR$ is not a $Q$-set$\}$;
\item $\mathfrak q=\min\{\kappa:$ each subset $X\subset \IR$ of cardinality $|X|\ge\kappa$ is not a $Q$-set$\}$.
\end{itemize}
It is clear that $\mathfrak q_0\le\mathfrak q$. Since each countable subset of the real line is a $Q$-set and each subset $A\subset\IR$ of cardinality continuum is not a $Q$-set, the cardinals $\mathfrak q_0$ and $\mathfrak q$ are uncountable and lie in the interval
$[\omega_1,\mathfrak c]$. So, these cardinals belong to small uncountable cardinals considered
in \cite{vD} and \cite{Vau}. It seems that for the first time the cardinals $\mathfrak q_0$ and $\mathfrak q$
appeared in the survey paper of J.~Vaughan \cite{Vau}, who refered to the paper \cite{GN}, which
 was not published yet at the moment of writing \cite{Vau}. Unfortunately, the cardinal
$\mathfrak q_0$ disappeared in the final version of the paper \cite{GN}. So, our initial motivation was to collect known information on the cardinal $\mathfrak q_0$ in order to have a proper reference (in particular, in the paper \cite{BKT} exploiting this cardinal). Studying this subject we have found a lot of interesting information on the cardinals $\mathfrak q_0$ and $\mathfrak q$ scattered in the literature. It seems that a unique paper devoted exclusively to the cardinal $\mathfrak q_0$ is \cite{Brend} of Brendle (who denotes this cardinal by $\mathfrak q$). Among many other results, in \cite{Brend} Brendle found a characterization of the cardinal $\mathfrak q_0$ in terms of weakly separated families.

Two families $\mathcal A$ and $\mathcal B$ of infinite subsets of a countable set $X$ are called
\begin{itemize}
\item {\em orthogonal} if $A\cap B$ is finite for every sets $A\in\A$ and $B\in\mathcal B$;
\item {\em weakly separated} if there is a subset $D\subset X$ such that $D\cap A$ is infinite for every $A\in\A$ and $D\cap B$ is finite for every $B\in\mathcal B$.
\end{itemize}
Let us recall that a family $\A$ of infinite sets is called {\em almost disjoint} if $A\cap B$ is finite for any distinct sets $A,B\in\A$.

\begin{theorem}\label{brendle} The cardinal $\mathfrak q_0$ is equal to the smallest cardinality $|A|$ of a subset $A\subset 2^\w$ such that the almost disjoint family $\A=\{B_x:x\in A\}$ of brances $B_x=\{x|n:n\in\w\}$ of the binary tree $2^{<\w}$ contains a subfamily $\mathcal B\subset\A$ which cannot be weakly separated from its complement $\A\setminus\mathcal B$.
\end{theorem}

Having in mind this characterization, let us consider the following two cardinals (introduced by Brendle in \cite{Brend}):
\begin{itemize}
\item $\mathfrak{ap}$, equal to  the smallest cardinality $|\A|$ of an almost disjoint family  $\A\subset[\w]^\w$ containing a subfamily $\mathcal B\subset\A$ which cannot be weakly separated from $\A\setminus\mathcal B$, and
\item $\mathfrak{dp}$, equal to the smallest cardinal $\kappa$ for which there exist two orthogonal families $\A,\mathcal B\subset[\w]^\w$ of cardinality $|\A\cup\mathcal B|\le\kappa$, which cannot be weakly separated.
\end{itemize}
 The notation $\mathfrak{dp}$ is an abbreviation of ``Dow Principle'' considered by Dow in \cite{Dow97}.

It is clear that $\mathfrak{dp}\le\mathfrak{ap}\le\mathfrak q_0\le\mathfrak q$.
In \cite{Brend} Brendle observed that the cardinals $\mathfrak{dp}$, $\mathfrak{ap}$, and $\mathfrak q_0$ are localized in the interval
$[\mathfrak p,\mathfrak b]$. Let us recall that $\mathfrak b$ is the smallest cardinality of a subset $B$ of the Baire
space $\w^\w$, which is not contained in a $\sigma$-compact subset of $\w^\w$.

The cardinal
$\mathfrak p$ is the smallest cardinality $|\F|$ of a family $\F$ of infinite subsets of $\w$
such that
\begin{itemize}
\item $\F$ has {\em strong finite intersection property}, which means that for each finite
subfamily $\mathcal E\subset\F$ the intersection $\cap\mathcal E$ is infinite, but
\item $\F$ has no infinite pseudo-intersection $I\subset \w$ (i.e., an infinite
set $I\subset\w$ such that $I\setminus F$ is finite for all $F\in\F$).
\end{itemize}

For a cardinal $\kappa$ its logarithm is defined as
$\log(\kappa)=\min\{\lambda:2^\lambda\ge\kappa\}$.
It is clear that $\log(\mathfrak c)=\w$ and $\log(\mathfrak c^+)\in[\w_1,\mathfrak c]$,
 so $\log(\mathfrak c^+)$ is a small uncountable cardinal.
 K\"onig's Lemma implies that $\log(\mathfrak c^+)\le\mathrm{cf}(\mathfrak c)$.
We refer the reader to  \cite{vD},
\cite{Vau} or \cite{Blass} for the definitions and basic properties of small uncountable cardinals
discussed in this note.

The following theorem collects some known lower and upper bounds on the cardinals $\mathfrak{dp}$, $\mathfrak{ap}$, $\mathfrak q_0$ and $\mathfrak q$. For the lower bound $\mathfrak p\le\mathfrak{dp}$ established in \cite{Brend} (and implicitly in \cite{Dow97}) we give an elementary proof which does not involve Bell's characterization \cite{Bell} of $\mathfrak p$ (as the smallest cardinal for which Martin's Axiom for $\sigma$-centered posets fails). Often the inequality $\mathfrak p\le\mathfrak q_0$ is attributed to Rothberger who actually proved in \cite{Roth} that $\mathfrak t>\w_1$ implies $\mathfrak q_0>\w_1$. It should be mentioned that $\mathfrak t=\mathfrak p$ according to a recent breakthrough result of Malliaris and Shelah \cite{MS}.

\begin{theorem}\label{main}
 $\mathfrak p\le\mathfrak{dp}\le\mathfrak{ap}\le\mathfrak q_0\le\min\{\mathfrak b,\mathfrak q\}\le\mathfrak q\le\log(\mathfrak c^+)$.
\end{theorem}

\begin{proof} The equality $\mathfrak q\le\log(\mathfrak c^+)$ follows from the fact that
each subset of a $Q$-set is Borel, and that a second countable space contains at
most $\mathfrak c$ Borel subsets.

The inequality $\mathfrak q_0\le\mathfrak q$ is trivial. To see that $\mathfrak q_0\le\mathfrak b$, choose any countable dense  subset $Q$ in the Cantor cube $2^\w$ and consider its complement $2^\w\setminus Q$, which is
 homeomorphic to the Baire space $\w^\w$ by the Aleksandrov-Urysohn Theorem \cite[7.7]{Ke}.
By the definition of the cardinal $\mathfrak b$, the space $2^\w\setminus Q$ contains a subset
$B$ of cardinality $|B|=\mathfrak b$ which is contained in no $\sigma$-compact subset of
$2^\w\setminus Q$. Then the union $A=B\cup Q$ is not a $Q$-set as the subset $B$ is not
relative $F_\sigma$-set in $A$.
Consequently, $\mathfrak q_0\le|B\cup Q|=|B|=\mathfrak b$.

The inequality $\mathfrak{ad}\le\mathfrak{q}_0$ follows from Theorem~\ref{brendle} and $\mathfrak{dp}\le\mathfrak{ap}$ is trivial.
Finally, we prove the inequality $\mathfrak p\le\mathfrak{dp}$. We need to check that any
two orthogonal families $\A,\mathcal B\subset[\w]^\w$ with $|\A\cup\mathcal B|<\mathfrak p$ are
weakly separated. By
$[\w]^{<\w}$ we denote the family of all finite subsets of $\w$.

For every $n\in\w$ and $x\in \A$ and $y\in \mathcal B$ consider
the families
$$\A_x=\{F\in[\w]^{<\w}:F\cap x=\emptyset\}\mbox{ \ and \ }
\mathcal B_{y,n}=\{F\in[\w]^{<\w}:|F\cap y|\ge n\}.$$
It is easy to check that the family
 $\F=\{\A_x:x\in \A\}\cup\{\mathcal B_{y,n}:y\in \mathcal B,\;n\in\w\}\subset[[\w]^{<\w}]^{\w}$
 has the strong finite
intersection property. Since $|\F|<\mathfrak p$, this family has an infinite pseudointersection
$\I=\{F_k\}_{k\in\w}$. It follows that the union $I=\bigcup_{k\in\w}F_k$ has finite intersection
 with each set $x\in \A$ and infinite intersection with each set $y\in \mathcal B$. Hence $\A$ and $\mathcal B$
are weakly separated.
\end{proof}

According to \cite{Fle78}, each $Q$-set $A\subset\IR$ is meager and Lebesgue null and hence belongs to the intersection $\M\cap\mathcal N$ of the ideal $\mathcal M$ of meager subsets of $\IR$ and the ideal $\mathcal N$ of Lebesgue null sets in $\IR$. The ideal $\mathcal M\cap\mathcal N$ contains the $\sigma$-ideal $\mathcal E$ generated by closed null sets in $\IR$. Cardinal characteristics of the $\sigma$-ideal $\mathcal E$ have been studied in \cite[\S2.6]{BaJu}. It turns out that each $Q$-set $A\subset \IR$ belongs to the ideal $\mathcal E$, which implies that $\mathfrak q_0\le\non(\mathcal E)$. More generally, $\mathfrak q_0\le \non(\mathcal I)$ for any flexible cccc $\sigma$-ideal $\mathcal I$ on $\IR$.

A family $\I$ of subsets of a set $X$ is called a {\em $\sigma$-ideal} on $X$ if $\bigcup\I=X\notin\I$ and for each countable subfamily $\A\subset\I$, each subset $A\subset\bigcup\A$ belongs to $\I$.
By $\non(\I)$ we denote the smallest cardinality $|A|$ of a subset $A\subset X$ which does not belong to $\I$. It is clear that $\w_1\le\non(\I)\le|X|$. We shall say that a $\sigma$-ideal on a topological space $X$ has ({\em $\sigma$-compact})  {\em Borel base} if each set $A\in\I$ is contained in a ($\sigma$-compact) Borel subset $B\in\I$ of $X$.

Let $\I$ be a $\sigma$-ideal on a set $X$.
A bijective function $f:X\to X$ will be called an {\em automorphism} of $\I$ if $\{f(A):A\in\I\}=\I$. It is clear that automorphisms of $\I$ form a
subgroup $\Aut(\I)$ in the group of all bijections of $X$ endowed with the operation of composition.  The group $\Aut(\I)$ will be called {\em the automorphism group} of the ideal $\I$. A $\sigma$-ideal $\I$ will be called {\em $\kappa$-flexible} for a cardinal number $\kappa$ if for any subsets $A,B\subset X$ with $|A\cup B|<\kappa$ there is an automorphism $f\in\Aut(\I)$ such that $f(A)\cap B=\emptyset$. A $\sigma$-ideal $\I$ on a set $X$ will be called {\em flexible} if it is $\kappa$-flexible for the cardinal $\kappa=|X|$.

\begin{proposition} Each $\sigma$-ideal $\I$ on any set $X$ is $\non(\I)$-flexible.
\end{proposition}

\begin{proof} Given any two subsets $A,B\subset X$ with $|A\cup B|<\non(\I)$, we need to find an automorphism $f\in\Aut(\I)$ such that $f(A)\cap B=\emptyset$. Since $|A\cup B|<\non(\I)\le|X|$, there is a subset $C\subset X\setminus (A\cup B)$ of cardinality $|C|=|A|$. Choose any bijective function $f:X\to X$ such that $f(A)=C$, $f(C)=A$ and $f$ is identity on the set $X\setminus (A\cup C)$. It is easy to see that $f$ is an automorphism of the $\sigma$-ideal $\I$ witnessing that $\I$ is $\non(\I)$-flexible.
\end{proof}

A $\sigma$-ideal $\I$ on a group $G$ is called {\em left-invariant} if $\{gA:A\in\I\}=\I$ for every $g\in G$ (which means that the automorphism group $\Aut(I)$ includes all left shifts $l_g:G\to G$, $l_g:x\mapsto gx$).

\begin{example} Each left-invariant $\sigma$-ideal $\I$ on a group $G$ is flexible.
\end{example}

\begin{proof} First we observe that the group $G\notin\I$ is uncountable. Then for any subset $A,B\subset G$ with $|A\cup B|<|G|$, the set $BA^{-1}=\{ba^{-1}:b\in B,\;a\in A\}$ has cardinality $|BA^{-1}|<|G|$. So we can find a point $g\in G\setminus BA^{-1}$ and observe that $gA\cap B=\emptyset$.
\end{proof}

We shall say that a $\sigma$-ideal $\I$ on a topological space $X$ satisfies the {\em compact countable chain condition} (briefly, $\I$ is a {\em cccc ideal\/}) if for any uncountable disjoint family $\mathcal C$ of compact subsets of $X$ there is a set $C\in\mathcal C$ that belongs to the ideal $\I$. This is a bit weaker than the classical {\em countable chain condition} (briefly, {\em ccc}) saying that for any uncountable disjoint family $\mathcal C$ of Borel subsets of $X$ there is a set $C\in\mathcal C$  belonging to the ideal $\I$. A simple example of a cccc $\sigma$-ideal that fails to have ccc is the $\sigma$-ideal $\K_\sigma$ of subsets of $\sigma$-compact sets in the Baire space $\w^\w$.

A metrizable space $X$ is called {\em analytic} if it is a continuous image of a Polish space (see \cite{Ke}).

\begin{proposition}\label{flex} Each $\mathfrak q_0$-flexible cccc $\sigma$-ideal $\I$ on an analytic space $X$ has $\non(\I)\ge \mathfrak q_0$.
\end{proposition}

\begin{proof} We need to show that any subset $A\subset X$ of cardinality $|A|<\mathfrak q_0$ belongs to the ideal $\I$. This is trivial if $|A|<\w_1$. So, we assume that $\w_1\le|A|<\mathfrak q_0$.

Using the $\mathfrak q_0$-flexibility of $\I$, by transfinite induction we can choose a transfinite sequence $(f_\alpha)_{\alpha\in\w_1}$ of automorphisms of $\I$ such that for every $\alpha\in\w_1$ the set $A_\alpha=f_\alpha(A)$ is disjoint with $\bigcup_{\beta<\alpha}f_\beta(A)$. The set $A_{\w_1}=\bigcup_{\alpha\in\w_1}A_\alpha$ has cardinality $|A_{\w_1}|=\max\{\w_1,|A|\}<\mathfrak q_0$ and hence is a $Q$-set (here we use the fact $Q$-sets are preserved by homeomorphisms and $A_{\w_1}$ being zero-dimensional, is homeomorphic to a subspace of the real line.) Consequently, for every $\alpha\in\w_1$ the subset $A_\alpha$ is of type $F_\sigma$ in $A_{\w_1}$ and we can find an $F_\sigma$-set $F_{\alpha}\subset X$ such that $F_\alpha\cap A_{\w_1}=A_\alpha$. It follows that for every $\alpha\in\w_1$ the set $B_\alpha=F_\alpha\setminus\bigcup_{\beta<\alpha}F_\beta$ is Borel in $X$, contains $A_{\alpha}$, and the family $(B_\alpha)_{\alpha\in\w_1}$ is disjoint. Each space $B_\alpha$ is analytic, being a Borel subset of the analytic space $X$. Consequently, we can find a surjective map $g_\alpha:\w^\w\to B_\alpha$ and choose a subset $A_\alpha'\subset\w^\w$ of cardinality $|A'_\alpha|=|A_\alpha|$ such that $g_\alpha(A_\alpha')=A_\alpha$. Since $|A'_\alpha|=|A_\alpha|<\mathfrak q_0\le\mathfrak b$, the set $A_\alpha'$ is contained in a $\sigma$-compact set $K_\alpha'\subset\w^\w$ according to the definition of the cardinal $\mathfrak b$. Then $K_\alpha=g_\alpha(K_\alpha')$ is a $\sigma$-compact subset of $B_\alpha$ containing the set $A_\alpha$. Since the family $(K_\alpha)_{\alpha\in\w_1}$ is disjoint and the $\sigma$-ideal $\I$ satisfies cccc, the set $\{\alpha\in\w_1:K_\alpha\notin\I\}$ is at most countable. So, for some ordinal $\alpha\in\w_1$ the set $K_\alpha$ belongs to $\I$ and so does its subset $A_\alpha$. Then $A=f_\alpha^{-1}(A_\alpha)\in\I$ as $f_\alpha\in\Aut(\I)$.
\end{proof}

Let $\fIcccc$ be the intersection of all flexible cccc $\sigma$-ideals on the real line.
Proposition~\ref{flex} implies that $\mathfrak q_0\le\non(\fIcccc)$.
So, any upper bound on the cardinal $\non(\fIcccc)$ yields an upper bound on $\mathfrak q_0$.

In fact, in the definition of the cardinal $\non(\fIcccc)$ we can replace the real line by any uncountable zero-dimensional Polish space. Given a topological space $X$ denote by $\fIcccc(X)$ the intersection of all flexible cccc $\sigma$-ideals on $X$.

\begin{proposition}\label{pol} For any uncountable Polish space $X$ we get $\non(\fIcccc)\le \non(\fIcccc(X))$. If the space $X$ is zero-dimensional, then $\non(\fIcccc)=\non(\fIcccc(X))$.
\end{proposition}

\begin{proof} Choose a subset $A\subset X$ of cardinality $|A|=\non(\fIcccc(X))$ that does not belong to the ideal $\fIcccc(X)$ and hence does not belong to some $\mathfrak c$-flexible cccc $\sigma$-ideal $\I$ on $X$. Let $X'$ be the (closed) subset of $X$ consisting of all points $x\in X$ that have no countable neighborhood in $X$. It follows that the space $X'$ is perfect (i.e., has no isolated points) and the complement $X\setminus X'$ is countable and hence belongs to the ideal $\I$. Fix any countable dense subset $D\subset X'$ and observe the space $Z=X'\setminus D$ is Polish and nowhere locally compact. By \cite[7.7, 7.8]{Ke}, the space $Z$ is the image of the space of irrationals $\IR\setminus\IQ$ under an injective continuous map $f:\IR\setminus\IQ\to Z$. It can be shown that $\mathcal J=\{A\subset\IR:f(A\setminus \IQ)\in\I\}$ is a $\mathfrak c$-flexible cccc $\sigma$-ideal on $\IR$ such that $f^{-1}(A)\notin\mathcal J$. So, $\non(\fIcccc)\le\non(\mathcal J)\le|f^{-1}(A)|\le|A|=\non(\fIcccc(X))$.

If the space $X$ is zero-dimensional, then by \cite[7.7]{Ke} the space $Z$ is homeomorphic to $\IR\setminus\IQ$ and we can assume that $f:\IR\setminus\IQ\to Z$ is a homeomorphism. Since the complement $X\setminus Z$ is countable, for every $\mathfrak c$-flexible cccc $\sigma$-ideal $\I$ on $\IR$ the family $f(\I)=\{A\subset X:f^{-1}(A)\in\I\}$  is a $\mathfrak c$-flexible cccc $\sigma$-ideal on $X$, which implies that $\non(\fIcccc(X))\le\non(\fIcccc)$.
\end{proof}

For a Polish group $G$ let $\Iccc(G)$ be the intersection of all
invariant ccc $\sigma$-ideals with Borel base on $G$. It is clear that $\fIcccc(G)\subset \Iccc(G)$ and hence $\non(\fIcccc(G))\le\non(\Iccc(G))$. For the compact Polish group $\IZ_2^\w=\{0,1\}^\w$ the ideal $\Iccc(\IZ_2^\w)$, denoted by $\Iccc$, was introduced and studied by Zakrzewski \cite{Zak1}, \cite{Zak2} who proved that $\mathfrak s_\w\le\non(\Iccc)\le \min\{\non(\M),\non(\mathcal N)\}$. Here $\mathfrak s_\w$ is the $\w$-splitting number introduced in \cite{Mal} and studied in \cite{CK}, \cite{Kra}. We recall that
\begin{itemize}
\item $\mathfrak s$ is the smallest cardinality of a subfamily $\mathcal S\subset[\w]^\w$ such that for every infinite set $X\subset\w$ there is $S\in\mathcal S$ such that $|X\cap S|=|X\setminus S|=\w$;
\item $\mathfrak s_\w$ is the smallest cardinality of a subfamily $\mathcal S\subset[\w]^\w$ such that for every sequence $(X_n)_{n\in\w}$ of infinite subsets of $\w$ there is $S\in\mathcal S$ such that $|X_n\cap S|=|X_n\setminus S|=\w$.
\end{itemize}
It is clear that $\mathfrak s\le\mathfrak s_\w$. On the other hand, the consistency of $\mathfrak s<\mathfrak s_\w$ is an open problem (attributed to Steprans). By Theorems 3.3 and 6.9 \cite{Blass}, the cardinal $\mathfrak s$ is localized in the interval $[\mathfrak h,\mathfrak d]$, where $\mathfrak h$ is the {\em distributivity number} and $\mathfrak d$ is the {\em dominating number} (it is equal to the smallest cardinality of a cover of $\w^\w$ by compact subsets). The proof of the inequality $\mathfrak s\le\mathfrak d$ in Theorem 3.3 of \cite{Blass} can be easily modified to obtain $\mathfrak s_\w\le\mathfrak d$. In the following theorem by $\mathcal E$ we denote the $\sigma$-ideal generated by closed Lebegue null sets on the real line.

\begin{theorem} $\mathfrak q_0\le\non(\fIcccc)\le\min\{\mathfrak b,\non(\Iccc)\}\le \min\{\mathfrak b,\non(\mathcal N)\}=\min\{\mathfrak b,\non(\mathcal E)\}$.
\end{theorem}

\begin{proof} The inequality $\mathfrak q_0\le\non(\fIcccc)$ follows from Proposition~\ref{flex}. Since $\fIcccc(\IZ^\w_2)\subset \Iccc(\IZ_2^\w)=\Iccc$, Proposition~\ref{pol} guarantees that $\non(\fIcccc)=\non(\fIcccc(\IZ^\w_2))\le\non(\Iccc)$. Observe that the $\sigma$-ideal $\K_\sigma$ consisting of subsets of $\sigma$-compact sets in the topological group $\IZ^\w$ is a flexible cccc $\sigma$-ideal with $\non(\K_\sigma)=\mathfrak b$. Then Proposition~\ref{pol} implies that $\non(\fIcccc)=\non(\fIcccc(\IZ^\w))\le\non(\K_\sigma)=\mathfrak b$. The inequality $\non(\Iccc)\le\min\{\non(\M),\non(\mathcal N)\}$ follows from the fact that the ideals $\M$ and $\N$ are ccc $\sigma$-ideals with Borel base. Taking into account that $\mathfrak b\le\non(\M)$, we conclude that $\min\{\mathfrak b,\non(\Iccc)\}\le\min\{\mathfrak b,\non(\M),\non(\N)\}=\min\{\mathfrak b,\non(\mathcal N)\}$. The equality $\min\{\mathfrak b,\non(\mathcal N)\}=\min\{\mathfrak b,\non(\mathcal E)\}$ follows from Theorem 2.6.8 \cite{BaJu}.
\end{proof}

Next, we establish some consistent inequalities between the cardinals $\mathfrak q_0$, $\mathfrak q$
 and some other known small uncountable cardinals, in particular
$\mathfrak h$, $\mathfrak g$, $\mathfrak s$, $\mathfrak d$, $\mathfrak e$.
The definitions of these cardinals and provable relations between them can be found in \cite{Blass} and \cite{Vau}.
We shall also consider a relatively new cardinal $\mathfrak{lr}$, called the {\em linear refinement number}, and equal to the minimal cardinality $|\mathcal F|$ of a family $\mathcal F\subset[\w]^\w$ with strong finite intersection property that has no linear refinement. A family $\mathcal L\subset [\w]^\w$ is called a {\em linear refinement} of $\F$ if $\mathcal L$ is linearly ordered by the preorder $\subset^*$ and for every $F\in\F$ there is $L\in\mathcal L$ with $L\subset^* F$.
The linear refinement number $\mathfrak{lr}$ was introduced by Tsaban in  \cite{T} (with the ad-hoc name $\mathfrak{p}^*$) and has been thoroughly studied in \cite{MST}.

ZFC-inequalities between the cardinals $\mathfrak{dp}$, $\mathfrak{ap}$, $\mathfrak q_0$, $\mathfrak q$ and some other cardinal characteristics of the continuum are described in the following diagram (the inequality $\mathfrak{ap}\le\mathrm{cov}(\mathcal M)$ was proved by Brendle in \cite{Brend}):

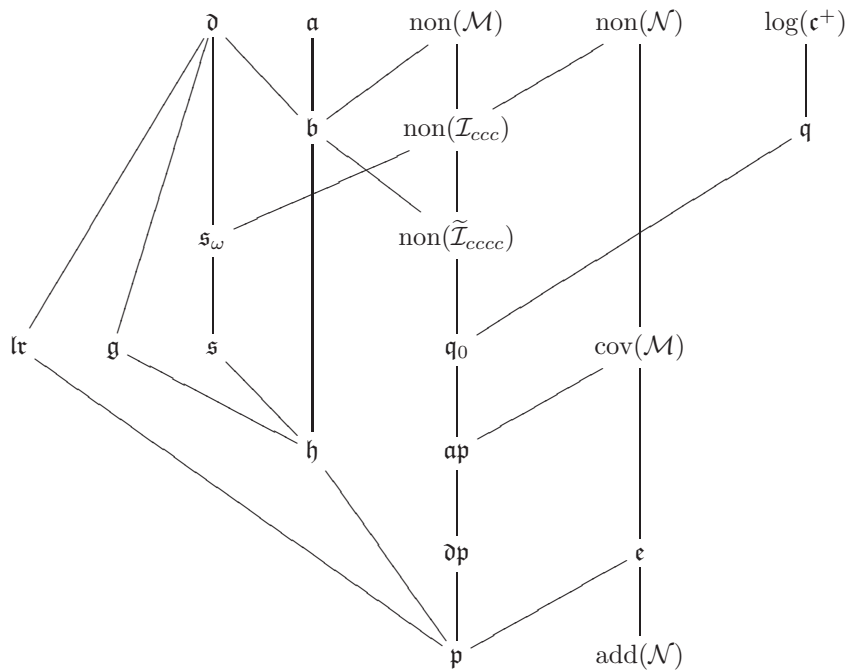
$$
\xymatrix{
&&\mathfrak d\ar@{-}[dd]\ar@{-}[llddd]\ar@{-}[lddd]\ar@{-}[rd]
&\mathfrak a\ar@{-}[d]&\non(\M)\ar@{-}[d]\ar@{-}[ld]&
\non(\mathcal N)\ar@{-}[ld]\ar@{-}[ddd]&\log(\mathfrak c^+)\ar@{-}[d]\\
&&&\mathfrak b\ar@{-}[ddd]\ar@{-}[rd]&\non(\I_{ccc})\ar@{-}[d]&&\mathfrak q\ar@{-}[lldd]\\
&&\mathfrak s_\w\ar@{-}[d]\ar@{-}[rru]&&\non(\widetilde\I_{cccc})\ar@{-}[d]\\
\mathfrak{lr}\ar@{-}[rrrrddd]&\mathfrak g\ar@{-}[rrd]
&\mathfrak s\ar@{-}[rd]&&\mathfrak q_0\ar@{-}[d]&\mathrm{cov}(\mathcal M)\ar@{-}[dd]\ar@{-}[ld]\\
&&&\mathfrak h\ar@{-}[rdd]&\mathfrak{ap}\ar@{-}[d]&\\
&&&&\mathfrak{dp}\ar@{-}[d]&\mathfrak{e}\ar@{-}[d]\ar@{-}[ld]\\
&&&&\mathfrak p&\mathrm{add}(\mathcal N)
}
$$
\medskip

The next theorem will be proved  mainly by combining known results (in a rather straightforward
way).

\begin{theorem} \label{th2} Each of the following inequalities is consistent with ZFC:
\begin{enumerate}
\item $\omega_1 =\mathfrak p= \mathfrak s= \mathfrak g= \mathfrak q_0=\mathfrak q=\log(\mathfrak c^+)<
\mathrm{add}(\mathcal N)=\mathfrak b=\mathfrak{lr}=\mathfrak c=\w_2$;
\item $\omega_1=\mathfrak q_0=\mathfrak q=\log(\mathfrak c^+)<\mathfrak h=\mathfrak{lr}=\mathfrak c=\omega_2$;
\item $\omega_1=\mathfrak p=\mathfrak s_{\w}<\mathfrak{dp}=\mathfrak q=\mathfrak c=\omega_2$;
\item $\w_1=\mathfrak q_0=\mathfrak q=\mathfrak b<\mathfrak g=\w_2$;
\item $\w_1=\mathfrak q_0=\mathfrak d=\non(\mathcal N)<\mathfrak q=\mathfrak c=\w_2$;
\item $\w_1=\mathfrak q_0=\mathrm{non}(\mathcal M)=\mathfrak{a}<\mathfrak{q}=\mathfrak{d}=\mathrm{cov}(\mathcal M)=\mathfrak c=\w_2$;
\item $\w_1=\mathfrak{dp}<\mathfrak{ap}=\mathfrak c=\w_2$;
\item $\w_1=\mathfrak{ap}<\mathfrak{q}_0=\mathfrak c=\w_2$;
\item $\w_1=\mathfrak p<\mathfrak{lr}=\w_2<\mathfrak q_0=\mathfrak c=\w_3$.
\end{enumerate}
\end{theorem}

\begin{proof}
1. The consistency of
$\w_1=\mathfrak s=\mathfrak p=\log(\mathfrak c^+)<\mathrm{add}(\mathcal N)=\mathfrak b=\mathfrak c=\w_2$
is a direct consequence of  \cite[Theorems~3.2 and 3.3]{JudShe88},
see \cite[Lemma 3.16]{JudShe88} for some explanations. The equality $\w_1=\mathfrak g$ follows from the
well-known  fact that
$\mathfrak g $ equals $\w_1$ after iterations with finite supports of Suslin posets,
see, e.g., \cite{Bre_slides}. The equality $\mathfrak{lr}=\w_2$ follows from Theorem 2.2 \cite{MST} (saying that $\mathfrak{lr}=\w_1$ implies $\mathfrak d=\w_1$).
\smallskip

2. To obtain a model of  $\w_1=\mathfrak q_0=\log(\mathfrak c^+)<\mathfrak h=\w_2$
let  us consider an iteration
 $\langle\mathbb P_\alpha,\dot{\mathbb Q}_\beta:\beta<\alpha\leq \w_2\rangle$
with countable supports such that $\mathbb Q_0$ is the countably closed Cohen poset adding
$\w_3$-many subsets to $\w_1$ with countable conditions.
For every $0<\alpha<\w_2$ let $\dot{\mathbb Q}_\alpha$ be a $\mathbb P_\alpha$-name
for the Mathias forcing, see \cite{Mat77} or \cite[p.~478]{Blass}. It is standard to check that
 $2^{\w_1}=\w_3>\w_2=2^\w$ holds in the final model, and
hence $\log(\mathfrak c^+)=\mathfrak q_0=\w_1$ there. Also, $\mathfrak h=\w_2=2^\w$ in this model simply by the design of the
Mathias poset, see the discussion on  \cite[p.~478]{Blass}. The equality $\mathfrak{lr}=\w_2$ follows from Theorem 2.2 \cite{MST}.
\smallskip

3. A model with $\w_1=\mathfrak p<\mathfrak{dp}=\mathfrak q=\w_2=\mathfrak c$ was
constructed by Alan Dow in \cite{Dow97}, see  Theorem 2 there. Below we shall also show
that {$\mathfrak s_{\w}=\w_1$} in that model.
Following \cite{BreHru09} we say that a forcing notion $\mathbb P$
\emph{strongly preserves countable tallness} if for every sequence
$\langle\tau_n:n\in\omega\rangle$  of $\mathbb P$-names for infinite
subsets of $\omega$
 there is a sequence $\langle B_n:n\in\omega\rangle$ of infinite subsets of $\omega$  such that for any
$B\in [\omega]^\omega$, if $B\cap B_n$ is infinite for all $n$, then
$\vDash_{\mathbb P}$ ``$B\cap\tau_n$ is infinite for all $n$''.
In \cite[Theorem~2]{Dow97} a poset $\mathbb P$ has been constructed such that
 $\mathfrak q_0=\mathfrak b=\mathfrak c>\omega_1$ holds in $V^{\mathbb P}$.
By the definition, $\mathbb P$ is an iteration with finite supports of
posets of the form $\mathbb Q_\mathcal{A}$, see \cite[Def.~2]{Dow97}.
Observe that the  notion of posets strongly preserving
countable tallness  remains the same if we demand the existence of
the sequence $\langle B_n:n\in\omega\rangle$ with the property stated there just for a
single $\mathbb P$-name $\tau $ for an infinite subset of $\omega$.
Therefore it follows from Lemmata 2,3 in \cite{Dow97} that the posets
$\mathbb Q_{\mathcal A}$ strongly preserve
countable tallness. Applying \cite[Lemma~5]{BreHru09} we conclude that
$\mathbb P$ strongly preserves countable tallness as well.
The latter easily implies that the ground model reals are splitting, and hence $\mathfrak s_\w=\omega_1$.
Indeed, given a {sequence of $\mathbb P$-names $\langle \tau_n:n\in\omega\rangle$ for an infinite subsets of $\omega$}
find an appropriate sequence $\langle B_n:n\in\omega\rangle$ of ground model
 infinite subsets of $\omega$. Now let $X\in [\omega]^\omega\cap V$
be such that $X$ splits all the $B_n$'s. Then $\Vdash$ ``$X$ splits every {$\tau_n$}''.
\smallskip

4. The condition $\w_1=\mathfrak b=\mathfrak{q}_0<\mathfrak g=\w_2=\mathfrak c$ holds, e.g., in the  model of
Blass and Shelah constructed in \cite{BS}, and in the Miller's model constructed in \cite{Mil84},
see   \cite{BL89} for the proof.
If, as in item 2, these forcings are preceded  by
the countably closed Cohen poset adding
$\w_3$-many subsets to $\w_1$ with countable conditions, then
we get in addition $2^{\w_1}=\w_3>\w_2=2^\w$ in the extension, and hence $\mathfrak q$ equals
$\w_1$ as well.
\smallskip

5. The consistency of $\w_1=\mathfrak q_0=\mathfrak d=\non(\mathcal N)<\mathfrak q=\mathfrak c=\w_2$ was proved
by Judah and Shelah \cite{JudSh91} (see also \cite{MilMAD}).
\smallskip

6. A model with $\w_1=\mathfrak q_0=\mathrm{non}(\mathcal M)=\mathfrak{a}<\mathfrak{q}=\mathfrak{d}=\mathrm{cov}(\mathcal M)=\mathfrak c=\w_2$ was constructed by Miller \cite{MilMAD}.

7,8. For every regular cardinal $\kappa>\w_1$ the consistency of the strict inequalities $\w_1=\mathfrak{dp}<\kappa=\mathfrak{ap}=\mathfrak c$ and $\w_1=\mathfrak{ap}<\kappa=\mathfrak{q}_0=\mathfrak c$ was proved by Brendle \cite{Brend}.

9. The consistency of $\w_1=\mathfrak p<\mathfrak{lr}=\w_2<\mathfrak q_0=\mathfrak c=\w_3$ follows from Theorem~\ref{machura} below.
\end{proof}

\begin{theorem}\label{machura}
Assume the Generalized Continuum Hypothesis and let  $\kappa$ and $\lambda$ be uncountable regular cardinal numbers such that $\kappa <\lambda = \lambda^{<\kappa}$.
There is a forcing notion $\mathbb{P}$ such that in a generic extention $V[G]$:
$\mathfrak{p}=\kappa$, $\mathfrak{lr} = \kappa^{+}$, and $\mathfrak{q}_0= \lambda= \mathfrak{c}$.
\end{theorem}

\begin{proof}
A forcing notion we use is very similar to one in  Theorem 3.9 from \cite{MST}. The difference is that we use Dow's focings $\mathbb{Q}_{\mathcal{A}}$ instead of Hechler forcing and a length of iteration is equal to the ordinal $\lambda \cdot \lambda$.

 More precisely the forcing $\mathbb{P}$ is given by an iteration:
\begin{enumerate}
\item $\seq{\mathbb{P}_{\alpha}, \dot{\mathbb{Q}}_{\beta}}{\alpha \leq \lambda \cdot \lambda, \beta < \lambda \cdot \lambda}$
is a finite support iteration;
\item $\mathbb{P}= \mathbb{P}_{\lambda \cdot \lambda}$;
\item $\mathbb{P}_0$ is the trivial forcing;
\item if  $\alpha = \lambda \cdot \xi $ where $\xi >0$, then:
\begin{enumerate}
\item$\Vdash_{\mathbb{P}_{\alpha}} \dot{\mathbb{Q}}_{\alpha}$ is  a Dow forcing $\mathbb{Q}_{\mathcal{A}_{\xi}}$ defined for a family $\dot{\mathcal{A}}_{\xi}$;
\item $\dot{\mathcal{A}}_{\xi}$ is a $\mathbb{P}_{\alpha}$-name for an ideal on $\w$ generated by an almost disjoint family of  cardinality $<\lambda$;
\item for each $\beta$ if
$\Vdash_{\mathbb{P}_{\beta}} \dot{\mathcal{A}}$  is  an ideal on $\w$ generated by an almost disjoint family of  cardinality $<\lambda$,
then exists  $\alpha > \beta$ such that $\alpha = \lambda \cdot \xi$ and
$\Vdash_{\mathbb{P}_{\alpha}} \dot{\mathcal{A}} = \dot{\mathcal{A}}_{\xi}$.
\end{enumerate}
\item  if $\alpha  \notin \set{\lambda \cdot \xi}{ \xi >0}$, then
\begin{enumerate}
\item $\Vdash_{\mathbb{P}_{\alpha}} \dot{\mathbb{Q}}_{\alpha}$ is an $\dot{\mathcal{F}}_{\alpha}$-Mathias forcing;
\item $ \dot{\mathcal{F}}_{\alpha}$ is a name for a filter generated by a family
$\set{\dot{A}_{\alpha,\iota}}{\iota < \iota_{\alpha}}$ which contains cofinite sets and has strong finite intersection property , where
$\iota_{\alpha}$ is an ordinal $< \kappa$;
\item $\iota_{\alpha}=0$ for $\alpha < \lambda$ (thus $\mathbb{Q}_{\alpha}$ is isomorphic to Cohen's forcing for $\alpha < \lambda$);
\item $\dot{A}_{\alpha,\iota}$ is a $\mathbb{P}_{\alpha}$-name for a subset of $\omega$;
\item $b_{\alpha, \iota} \colon (2^{\omega})^{\omega} \to [\w]^{\w} $ is a Borel function coded in the ground model;
\item $\Vdash_{\mathbb{P}_{\alpha}} \dot{A}_{\alpha, \iota} = b_{\alpha, \iota} ( \seq{\dot{B}_{\gamma ( \alpha, \iota, n )}}{n<\omega} )$, where $B_{\alpha}\subset [\w]^{\w}$ denotes the $\alpha$-th generic real;
\item If $\alpha = \lambda \cdot \xi + \nu$, then
$$ \gamma (\alpha, \iota, n) < \lambda \cdot \xi.$$
\item
For each $\zeta < \lambda$ and each sequence $\seq{b_{\iota}}{\iota < \iota_*}$ of Borel functions
$b_{\iota} \colon (2^{\omega})^{\omega} \to [\w]^{\w} $ of length $\iota_*<\kappa$,
and all ordinal numbers $\delta (\iota, n)<\lambda \cdot \zeta $
such that $\mathbb{P}$ forces that the filter generated
by the cofinite sets together with the family
$$\set{b_{\iota} (\seq{B_{\delta ( \iota, n )}}{n<\omega})}{\iota < \iota_*},$$
is proper,
there are arbitrarily large $\alpha < \lambda \cdot ( \zeta +1) $ such that:
\begin{enumerate}
\item $\iota_{\alpha}=\iota_*$;
\item $b_{\alpha, \iota} = b_{\iota} $ for all $\iota < \iota_*$;
\item $\gamma (\alpha, \iota, n) = \delta ( \iota, n)$ for all $\iota < \iota_*$ and all $n$.
\end{enumerate}
\end{enumerate}
\end{enumerate}

A proof of equalities $ \mathfrak{p}=\kappa$, $\mathfrak{lr} = \kappa^{+}$  is essentially the same as in Lemmata 3.11~--~3.15  in \cite{MST}. The only difference is in the iteration Lemma 3.10. Here we need to observe that Dow's forcings $\mathbb{Q}_{\mathcal{A}}$ cannot add a pseudointersection to a familiy with strong finite intersection property formed by Cohen's reals (more generally to  eventually narrow families). This was proven by Dow in Lemma 2 in \cite{Dow97}.

Adding Dow's forcings instead of Hechler forcing give us an inequality
 $\mathfrak{q}_0 \geq \lambda$ instead of $\mathfrak{b} \geq \lambda$.

\end{proof}

The argument in the remark  below is usually attributed to Devlin and Shelah \cite{DevShe78}.
We have learned it from David Chodounsky.
\medskip

\begin{remark}
We did not have   to start with the countably closed Cohen poset  adding
$\w_3$-many subsets to $\w_1$  in items 2 and 4
of  Theorem~\ref{th2} in order to guarantee that $\mathfrak q=\w_1$.
However, the argument presented in the proof of Theorem~\ref{th2} seems
to be easier and more direct, and hence we presented it for those readers who are interested
just in the consistency of corresponding constellations.

Following \cite{MooHruDza04} we denote by $\Diamond(2,=)$ the following statement:
\emph{For every Borel $F:\w^{<\w_1}\to 2$ there exists $g:\w_1\to 2$
such that for every $f:\w_1\to 2$ the set $\{\alpha:F(g\uhr\alpha)=f(\alpha)\}$
is stationary}. Here $F:\w^{<\w_1}\to 2$ is Borel iff $f\uhr\w^\alpha\to 2$ is Borel for
all $\alpha\in\w_1$.
$\Diamond(2,=)$ implies that $\mathfrak q=\w_1$, which means that no uncountable $Q$-set of reals exists. Indeed,
suppose $X = \{x_\alpha: \w<\alpha<\w_1\}$  is a $Q$-set of reals.
Choose some nice coding for $G_\delta$ sets of reals by elements of $2^\w$.
For each $\alpha\in (\w,\w_1)$  define $F_\alpha : 2^\alpha \to 2$ as follows:
For $x$ in $2^\alpha$ put $F_\alpha(x) = 1$ iff  $x_\alpha$ is in the $G_\delta$
set coded by $x\uhr\w$.  $F_\alpha$  is
Borel and thus $F = \bigcup_{\alpha\in\w_1} F_\alpha $  is also Borel. Therefore there exists
 a guessing function $g: \w_1\to 2$ for $F$. Put $Y = \{x_\alpha : g(\alpha) = 0\}$.
Then $Y$ is not a $G_\delta$  subset of $X$.
In order to show this choose a $G_\delta$  set $G$ and any $f: \w_1\to 2$ such that $f\uhr\w$ codes $G$.
Then there is $\beta$  such that
$F(f\uhr\beta) = g(\beta)$, and hence $x_\beta$  is in $G \Delta  Y$ which means $G\cap X \neq Y$.
Finally,  it suffices to note that
$\Diamond(2,=)$ holds in any model considered in items 2,4 of Theorem~\ref{th2},
see \cite[Theorem~6.6]{MooHruDza04}.
\end{remark}
%\medskip

It would be nice to know more about the relation of the cardinals $\mathfrak q_0$ and $\mathfrak q$ to
the cardinals $\mathfrak g$, $\mathfrak e$,
$\cov(\mathcal M)$, and $\cov(\mathcal N)$. Here $\mathfrak e$ is the {\em evasion number} considered by A.~Blass in \cite[\S10]{Blass}.
It follows from \cite[10.4]{Blass} that $\mathfrak q_0=\mathfrak b<\mathfrak e$ is consistent.

\begin{problem} Is any of the inequalities
 $\mathfrak q_0 >\cov(\mathcal M)$, $\mathfrak q_0>\mathfrak e$, $\mathfrak q_0>\mathfrak g$, $\non(\fIcccc)>\mathfrak q_0$ consistent?
In particular, what are the values of
  $\mathfrak e$ and $\mathfrak g$ in the model of Dow (or its modifications)?
\end{problem}

The question whether  $\mathfrak q_0 >\cov(\mathcal M)$ is consistent seems the most intriguing
among those mentioned above. In \cite{Brend} this question is attributed to A.~Miller.
%\newpage

\end{document}